\documentclass[11pt]{article}
\usepackage[utf8]{inputenc}
\usepackage[T1]{fontenc}
\usepackage{lmodern}
\usepackage{amsmath, amssymb, amsthm, amsfonts}
\usepackage{mathtools}
\usepackage{xcolor}
\usepackage{hyperref}
\usepackage{enumitem}
\usepackage{bbm}
\usepackage{geometry}
\usepackage{graphicx}
\usepackage{subcaption}
\usepackage{microtype}

\theoremstyle{definition}
\newtheorem{theorem}{Theorem}[section]
\newtheorem{lemma}[theorem]{Lemma}
\newtheorem{proposition}[theorem]{Proposition}
\newtheorem{corollary}[theorem]{Corollary}
\newtheorem{definition}[theorem]{Definition}

\usepackage{hyperref}
\makeatletter
\renewenvironment{proof}[1][\proofname]{%
	\par\pushQED{\qed}%
	\normalfont                               
	\topsep6\p@\@plus6\p@\relax               
	\trivlist
	\item[\hskip\labelsep\bfseries #1\@addpunct{.}]\ignorespaces
}{%
	\popQED
	\endtrivlist
	\@endpefalse
}
\makeatother
\newcommand{\indic}[1]{\mathbf{1}_{#1}}
\newcommand{\fa}{a}
\newcommand{\tmin}{t_{\min}}
\newcommand{\tmax}{t_{\max}}
\title{\textbf{Preventive Maintenance of a Two-Unit Priority Standby System with Repair}}
\author{Alexandros Carballo\thanks{Universidad de La Habana, alexandros.carballo@estudiantes.matcom.uh.cu}
\and 
Marelys Crespo Navas\thanks{Université Paris-Est Créteil, marelys.crespo-navas@u-pec.fr}
\and
José E. Valdés\thanks{Universidad de La Habana, vcastro@matcom.uh.cu} }
\date{}
\usepackage[normalem]{ulem}
\begin{document}

\maketitle

\begin{abstract}
Optimal maintenance policies play an important role in the reliability analysis of repairable systems. This paper examines a two-unit priority standby system with a repair facility, where the priority unit is subject to preventive maintenance. In the general case, we derive necessary and sufficient conditions under which maintenance increases the mean time to system failure. When the hazard rate of the priority unit has either a bathtub-shaped or an upside-down bathtub-shaped form, we establish conditions for the existence of both a threshold time, beyond which maintenance improves the mean time to system failure, and the optimal maintenance time. Expressions to find these quantities are provided in terms of the hazard rate and the mean residual life function. Furthermore, stochastic ordering techniques are used to compare two independent systems under this model, along with their respective threshold times and optimal maintenance times. 
\end{abstract}

\section{Introduction} \sloppy

Standby redundant systems with repair are widely used to enhance the reliability and availability of engineering systems. In such a setting, preventive maintenance policies can significantly extend system lifetime and reduce the occurrence of unexpected failures. Determining whether such maintenance should be implemented, as well as identifying the optimal maintenance time, is a central problem in reliability theory. A comprehensive overview of maintenance theory is provided in \cite{nakagawa2005maintenance}. More recent advances on reliability and maintenance, along with numerous references therein, can be found in \cite{kimura2023reliability}, \cite{liu2023} and \cite{zhao2025reliability}.

Within this research area, several studies have analyzed different models for the preventive maintenance of standby systems with repair. Our research  focuses on a two-unit standby priority system with repair and maintenance. This priority configuration, although without maintenance, has been studied before, see \cite{Nakagawa1975}, and more recently  \cite{li2024time,eryilmaz2022reliability}. Other studies have instead focused on two-unit standby systems with maintenance but without priority, see, for instance, \cite{ZHONG2014405,mahmoud2010two}, as well as the pioneering works \cite{nakagawa1974optimum,osaki1970two}. Only a limited number of works have addressed preventive maintenance in two-unit priority standby systems with repair; one such study is \cite{HIRATA2019183}, where the aim and the method used are very different from our approach.

To the best of our knowledge, preventive maintenance policies for priority standby systems with non-monotonic hazard rates for the main unit have not been previously investigated, and no stochastic comparison results based on stochastic orders are available for this model.

Motivated by these gaps, we study a two-unit priority standby configuration with a repair facility, where the priority unit is subject to preventive maintenance. In a general setting, we derive conditions under which maintenance improves the mean time to system failure in terms of the mean residual life function of the priority unit. Particular attention is devoted to systems whose priority unit exhibits bathtub-shaped or upside-down bathtub-shaped hazard rates, for which we characterize both the threshold time beyond which maintenance improves the mean time to system failure and the optimal maintenance time. Expressions to find these quantities are provided in terms of the hazard rate and the mean residual life function. Numerical examples are presented to illustrate the impact of the hazard rate shape on preventive maintenance decisions. In addition, stochastic comparisons between two independent configurations are developed.

The remainder of the paper is organized as follows. Section \ref{Prelim} presents the main definitions and preliminary results. Section \ref{Sec 2} presents the model and derives its mean time to failure. Section \ref{Sec 3} studies the effect of preventive maintenance and characterizes threshold and optimal maintenance times under different hazard rate structures. Section \ref{Sec 4} provides numerical illustrations of the theoretical results. Section \ref{Sec 5} develops stochastic comparisons between independent configurations under the proposed model. Section \ref{coclusiones} concludes.

\section{Definitions and basic results} \label{Prelim} 

We next introduce the classes of distributions considered in this work and the stochastic orders that will be used in the analysis.

\subsection{Aging classes and bathtub-type distributions} 

Throughout this work, all random variables are assumed to be nonnegative. Moreover, the terms \textit{increasing} and \textit{decreasing} are used in a non-strict sense; that is, increasing means non-decreasing and decreasing means non-increasing.

Consider a random variable $X$ representing a lifetime, with density function $f$ and distribution function $F$. Its survival function is given by $\overline{F}(t)= 1-F(t)$, for $t\geq 0$. The residual lifetime of $X$ at time $t$ is defined as $X_t = (X-t \mid X>t)$. 

To characterize the aging behavior of $X$, we consider two important functions in reliability theory. The first is the hazard rate, defined by 
\[
\displaystyle r(t) = \frac{f(t)}{\overline{F}(t)},
\]
whenever $\overline{F}(t)>0$. The second is the mean residual life function (MRL), defined by 
\[
\displaystyle m(t) = \mathbb{E}\left[X_t\right] = \frac{1}{\overline{F}(t)} \displaystyle \int_t^\infty \overline{F}(x)\, \mathrm{d}x,
\]
provided that $\mathbb{E}[X]<\infty$ with $\overline{F}(t)>0$.

When these two functions exhibit monotonic behavior, they give rise to well-known aging classes; see the classical reference \cite{Barlow1975}. In particular, if the hazard rate is increasing, the distribution is said to belong to the $\mathrm{IFR}$ class, denoted $X\in\mathrm{IFR}$; if the mean residual life function is decreasing, it belongs to the DMRL class, denoted $X\in\mathrm{DMRL}$. Conversely, anti-aging classes arise when the opposite monotonicity properties hold, namely the $\mathrm{DFR}$ (decreasing failure rate) and $\mathrm{IMRL}$ (increasing mean residual life) classes. Moreover, the inclusions $\mathrm{IFR} \subset \mathrm{DMRL}$ and $\mathrm{DFR} \subset \mathrm{IMRL}$ are satisfied.

Although much of the literature has focused on monotonic aging, many systems exhibit non-monotonic aging behavior, typically modeled by distributions with bathtub-shaped hazard rates. A comprehensive review of such distributions is provided in \cite{LAI200169}. Since multiple definitions exist, the following one is adopted in this work.

\begin{definition}[Bathtub hazard rate] \label{BFR} Let $X$ be a random variable with hazard rate $r$, continuous and well defined for all $t\geq 0$. We say that $X$ has a bathtub-shaped hazard rate, and we write $X \in \mathrm{BFR}$, if there exists a value $\tmin \geq 0$ such that $r(t)$ is decreasing for $0\leq t\leq \tmin$, and increasing for $t\geq \tmin$. If $\tmin=0$, then $r(t)$ is increasing for all $t \geq 0$, in which case $X \in \mathrm{IFR}$. 
\end{definition} 

Analogously, it is natural to consider distributions whose hazard rate exhibits an upside-down bathtub shape. 

\begin{definition}[Upside-down bathtub hazard rate]
Let $X$ be a random variable with hazard rate $r$, continuous and well defined for all $t\geq 0$. We say that $X$ has an upside-down bathtub-shaped hazard rate, and we write $X \in \mathrm{UBFR}$ if there exists a value $\tmax > 0$ such that $r(t)$ increases for $0\leq t \leq \tmax$ and decreases for $t\geq \tmax$.
\end{definition}

Note that if $\tmax = 0$, then the hazard rate $r(t)$ is decreasing for all $t \geq 0$, which means that $X$ belongs to the class $\mathrm{DFR}$. In this sense, the DFR class can be viewed as a particular case of UBFR distributions. However, this situation will be analyzed separately, and whenever we refer to UBFR distributions, we implicitly assume that $\tmax > 0$.

Several well-known distributions exhibit an upside-down bathtub-shaped hazard rate, including the lognormal and certain Pareto-type distributions. From a practical point of view, this type of behavior describes a unit that is more likely to fail at the beginning, for instance, due to initial defects or adjustment effects. As time goes on, the hazard rate first increases as these issues manifest, reaches a peak, and then decreases as the unit stabilizes and operates under more regular conditions.

\subsection{Stochastic orders}\label{Stochastic orders}
Stochastic orders provide a useful framework for comparing the performance of stochastic systems under different probabilistic criteria; see, for example, \cite{belzunce2015introduction,kochar2022stochastic}. Below, we introduce several orders that will be used throughout the paper. For a detailed exposition, the reader is referred to \cite{Shaked2007}, a fundamental reference in this area.

\begin{definition}[Stochastic orders] \label{stoch ord} Let $X$ and $Y$ be two non-negative random variables with survival functions $\overline{F}$ and $\overline{G}$, and well-defined mean residual life functions $m_X$ and $m_Y$. We say that $X$ dominates $Y$ in the
\begin{enumerate}
\item[$a)$] Usual stochastic order, denoted by $X\geq_{\mathrm{st}} Y$, if $\overline{F}(t) \geq \overline{G}(t),$ for all $t\geq 0$.

\item[$b)$] Laplace transform order, denoted by $X\geq_{\mathrm{Lt}} Y$, if 
\[ 
\widehat{\overline{F}}(s) = \int_{0}^{+\infty} \mathrm{e}^{-st}\,\overline{F}(t)\,\mathrm{d}t \geq   \int_{0}^{+\infty} \mathrm{e}^{-st}\,\overline{G}(t)\,\mathrm{d}t = \widehat{\overline{G}}(s), \quad \text{for all } s\geq 0. \]
The functions $\widehat{\overline{F}}$ and $\widehat{\overline{G}}$ are the Laplace transforms of $\overline{F}$ and $\overline{G}$, respectively.

\item[$c)$] Mean residual life order, denoted by $X\geq_{\mathrm{mrl}} Y$, if $m_X(t) \geq m_Y(t)$, for all $t\geq 0$, whenever the expectations exist. 
\end{enumerate}
\end{definition}

It is well known that the usual stochastic order implies the Laplace transform order. However, in general, the mean residual life order is not comparable with the other two orders.

A useful characterization of the usual stochastic order is the following: $X\geq_{\mathrm{st}} Y$ if and only if
\begin{equation}\label{st}
\mathbb{E}\left[\phi(X)\right] \geq \mathbb{E}\left[\phi(Y)\right],
\end{equation}
for every increasing function $\phi$ such that the expectations exist.

\section{System description and mean time to failure}\label{Sec 2}

Let us consider a system with two units and one repair facility. Unit~1 is referred to as the priority (or main) unit, while unit~2 is the standby unit. The priority unit operates whenever it is available, whereas the standby unit remains in a cold standby state and becomes operative only when the priority unit is unavailable due to failure or maintenance.

At the initial time, the priority unit begins operation, while the standby unit remains in cold standby. A preventive maintenance policy is applied to the priority unit, under which maintenance is scheduled after a predetermined time $\mathrm{T}>0$ from the start of operation. If the unit fails before the scheduled maintenance epoch, the preventive maintenance is not performed.

If the standby unit fails while the priority unit is still under repair or maintenance, a system failure occurs. Otherwise, once the repair or maintenance activity is completed, the priority unit resumes operation and the standby unit returns to its standby state. This operational cycle continues until the system failure. All repair actions restore the failed unit to an ``as good as new'' condition, and all switching operations between units are assumed to be perfect and instantaneous.

The lifetimes of the priority and standby units are denoted by $X_1$ and $X_2$, respectively. The distribution function, survival function, hazard rate, mean residual life function, and finite expectation of $X_1$ are denoted by $F$, $\overline{F}$, $r$, $m$, and $\fa$, respectively. We will assume that $F$ is differentiable, that $F(0)=0$, and that $\overline{F}(t)>0$ holds for all $t>0$.

The standby unit is assumed to follow an exponential distribution with a hazard rate $\lambda>0$. The repair and preventive maintenance times of the priority unit are denoted by $Y_1$ and $Y_2$, with distribution functions $G_1$ and $G_2$, respectively. We will assume that $G_i(0)=0$ for $i=1,2$. The lifetimes, repair times, and preventive maintenance times are assumed to be mutually independent.

In this paper, the terms \emph{time to system failure} and \emph{system lifetime} are used interchangeably, both denoting the random variable representing the duration from system start-up until failure. Formally, we write $\tau(\mathrm{T})$ to denote this quantity. For each cycle $i\ge1$, assume that $X_{ji}=_{\mathrm{st}}X_j$ and $Y_{ji}=_{\mathrm{st}} Y_j$ for $j=1,2$, and that all these random variables are mutually independent. Let $\nu$ denote the number of cycles completed before system failure.

From the description of the model, the time to system failure can be expressed as
\[
\tau(\mathrm{T}) = \sum_{i=1}^{\nu} \left( X_{1i} \wedge \mathrm{T} + ( X_{2i} \wedge Y_{1i} ) \,\mathbbm{1}_{\{X_{1i}\leq \mathrm{T} \}} + ( X_{2i} \wedge Y_{2i})\,\mathbbm{1}_{\{X_{1i}>\mathrm{T} \}} \right),
\]
where each summand represents the duration of the $i$-th operating cycle. In the expression above, $\wedge$ denotes the minimum operator, and $\mathbbm{1}$ denotes the indicator function. The random variable $\nu$ follows a geometric distribution with parameter
\[
p(\mathrm{T}) = \mathbb{P}\left(X_2 < Y_1\mathbbm{1}_{\{X_1 \leq \mathrm{T}\}} + Y_2 \,\mathbbm{1}_{\{X_1>\mathrm{T}\}} \right) = \lambda \left( \mu_1 - \Delta \mu \,\overline{F}(\mathrm{T})\right),
\]
where, for $i=1,2$, $\mu_i = \mathbb{E}[X_2\wedge Y_i]$, and $\Delta \mu = \mu_1 - \mu_2$.

Since the operating cycles are independent and identically distributed, and the stopping time $\nu$ depends only on the outcomes of the preceding cycles, Wald's identity yields the mean time to system failure
\begin{equation} \label{mean time - priority and PM system}
\mathcal{M}(\mathrm{T}) = \mathbb{E}\!\left[\tau(\mathrm{T})\right] = \frac{1}{\lambda}\left(1 + \frac{\int_{0}^{\mathrm{T}} \overline{F}(t)\,\mathrm{d}t}{\mu_1 - \Delta \mu \,\overline{F}(\mathrm{T})}\right).
\end{equation}

When no preventive maintenance is performed, the mean time to failure is given by
\begin{equation}\label{mean time - priority system}
\mathcal{M}(\infty) = \frac{1}{\lambda}\left(1+\frac{\fa}{\mu_1 }\right),
\end{equation}
which is a well-known result in reliability theory; see, for instance, \cite{Nakagawa1975}.

\section{Threshold time and optimal maintenance time}\label{Sec 3}

Prior to analyzing the value of $\mathrm{T}$ that maximizes the mean time to system failure, we derive necessary and sufficient conditions that guarantee preventive maintenance produces a strict improvement, that is,
\[
\mathcal{M}(\mathrm{T})>\mathcal{M}(\infty).
\]

Maintenance times are commonly assumed to be shorter than repair times, in the sense that $Y_1\geq_{\mathrm{st}} Y_2$, that is, $G_1(t)\leq G_2(t)$ for all $t\geq 0$; see, for example, \cite{nakagawa2005maintenance}. However, the following proposition shows that the weaker condition $\Delta \mu>0$ is necessary for preventive maintenance to be beneficial.

\begin{proposition}
 Let $\mathrm{T} > 0$. A necessary condition for $\mathcal{M}(\mathrm{T}) > \mathcal{M}(\infty)$ is $\Delta \mu > 0$.
\end{proposition}
\begin{proof}
From \eqref{mean time - priority and PM system} and \eqref{mean time - priority system}, the inequality  $\mathcal{M}(\mathrm{T}) > \mathcal{M}(\infty)$ is equivalent to
\begin{align} \label{eq: beneficial maintenance}
	\mu_1 \int_{\mathrm{T}}^\infty \overline{F}(t)\,\mathrm{d}t < \fa \,\Delta\mu\, \overline{F}\left(\mathrm{T}\right). 
\end{align}
Since the left-hand side of \eqref{eq: beneficial maintenance} is strictly positive, this inequality can hold only if $\Delta \mu > 0$. 
\end{proof}

From now on, we write 
\begin{equation}\label{formule K}
K= \frac{\fa \, \Delta \mu}{\mu_1}.
\end{equation}
In the next proposition, we provide a necessary and sufficient condition for the mean time to system failure under preventive maintenance to exceed that of the system without maintenance. It follows that \eqref{eq: beneficial maintenance} can be equivalently rewritten as
\[
m(\mathrm{T}) < K.
\]
We then state the result formally.
\begin{proposition}\label{prop: beneficial maintenance}
Let $\mathrm{T}>0$ and $\Delta \mu>0$. Then $\displaystyle \mathcal{M}(\mathrm{T})> \mathcal{M}(\infty)$ if and only if $m(\mathrm{T}) < K$.
\end{proposition}

When the mean residual life of the main unit is increasing, the priority unit exhibits an anti-aging pattern, and the implementation of a preventive maintenance policy reduces the mean time to system failure. This result is established in the following proposition. In particular, preventive maintenance is non-beneficial when $X_1\in\mathrm{DFR}$. 

\begin{proposition}\label{prop: imrl}
Let $X_1\in\mathrm{IMRL}$. Then $\displaystyle \mathcal{M}(\mathrm{T})<\mathcal{M}(\infty)$, for any $\mathrm{T}> 0$.
\end{proposition}

\begin{proof}
Condition $X_1\in\text{IMRL}$ implies that, for any $\mathrm{T}>0$, 
\[
m(\mathrm{T}) \geq m(0) = \fa >  K.
\]
For this reason,
\[
\mu_1 \int_{\mathrm{T}}^\infty \overline{F}(t)\,\mathrm{d}t > \fa\, \Delta\mu\, \overline{F}\left(\mathrm{T}\right),
\]
and therefore $\mathcal{M}(\mathrm{T})<\mathcal{M}(\infty)$. This result holds independently of the sign of $\Delta \mu$.
\end{proof}

We now investigate the existence of a threshold time $\mathrm{T}_0$, for which preventive maintenance increases the expected system lifetime:
\[
\mathrm{T}_0 = \inf\left\{ \mathrm{T}>0 : \mathcal{M}(\mathrm{T})>\mathcal{M}(\infty) \right\}.
\]
By definition, preventive maintenance is not beneficial for $\mathrm{T} < \mathrm{T}_0$. However, the condition $\mathrm{T} > \mathrm{T}_0$ does not ensure a positive effect; $\mathrm{T}_0$ merely defines the lower bound of the set of times for which maintenance may become beneficial.

Using Proposition \ref{prop: beneficial maintenance}, we obtain the following result.

\begin{proposition}\label{prop : T_0 general}
Let $\Delta \mu>0$. Then $\mathrm{T}_0>0$ exists if and only if $\displaystyle \inf_{\mathrm{T}>0} m(\mathrm{T}) < K$. In that case,
\begin{align}\label{formule_T0}
\mathrm{T}_0=\inf\left\{\mathrm{T}>0 : m(\mathrm{T})<K\right\}.
\end{align}
\end{proposition}

\begin{proof}
By Proposition \ref{prop: beneficial maintenance},
\[
\mathrm{T}_0=\inf\left\{\mathrm{T}>0 : \mathcal{M}(\mathrm{T})>\mathcal{M}(\infty)\right\}
=\inf\left\{\mathrm{T}>0 : m(\mathrm{T})<K\right\}.
\]
Hence, $\mathrm{T}_0$ exists if and only if the set $A_K=\{\mathrm{T}>0 : m(\mathrm{T})<K\}$ is non-empty, which is equivalent to
\[
\inf_{\mathrm{T}>0} m(\mathrm{T}) < K .
\]

Finally, since $m$ is continuous and $m(0)=\fa>K$, the inequality $m(\mathrm{T})<K$ cannot hold for $\mathrm{T}$ arbitrarily close to $0$. Therefore, $\mathrm{T}_0=\inf A_K>0$.
\end{proof}

Assuming that maintenance is beneficial, that is, that there exists $\mathrm{T}_0>0$, we now examine the existence of an optimal maintenance time, denoted by $\mathrm{T}^*$, and defined as 
\begin{equation}\label{eq: formule T*}
\mathrm{T}^{*} = \arg\max_{\mathrm{T}>\mathrm{T}_0} \mathcal{M}(\mathrm{T}). 
\end{equation}
In particular, we consider two cases in which the hazard rate $r$ is bathtub-shaped and upside-down bathtub-shaped, respectively.

\subsection{Bathtub-shaped hazard rate}\label{Sec 3.1}

We recall the following result from \cite{Olcay1995}. In that work, the author adopts a different definition of the bathtub hazard rate, assuming strict monotonicity. The lemma, however, can be established analogously under Definition \ref{BFR}.

\begin{lemma}[Theorem 1 in \cite{Olcay1995}]\label{Lemma Olcay BFR}
Let $X_1 \in \mathrm{BFR}$.
\begin{enumerate}
	\item[$a)$] If $r(0)\leq \frac{1}{\fa}$, then $X_1\in\mathrm{DMRL}$.
	\vspace{0em}
	\item[$b)$] If $r(0)>\frac{1}{\fa}$, then there exists $c\in(0,\tmin)$ such that $m$ is increasing on $[0,c]$ and decreasing on $[c,\infty)$.
\end{enumerate}
\end{lemma}

Lemma \ref{Lemma Olcay BFR} shows that a bathtub-shaped hazard rate imposes a specific structure on the mean residual life function, depending on the value of $r(0)=f(0)$. By combining this result with Propositions \ref{prop: beneficial maintenance} and \ref{prop : T_0 general}, one can derive a necessary and sufficient condition for the existence of $\mathrm{T}_0$. Moreover, for any $\mathrm{T}>\mathrm{T}_0$, performing maintenance at time $\mathrm{T}$ yields an increase in the expected lifetime of the system.

\begin{proposition} \label{prop : T_0}
Let $\displaystyle \Delta \mu > 0$ and $X_1\in\mathrm{BFR}$. Then $\mathrm{T}_0>0$ exists if and only if	
\begin{align}\label{condition T0}
	\lim\limits_{t\rightarrow\infty} m(t) < K \quad  \text{or, equivalently,}\quad \lim\limits_{t\rightarrow \infty}	r(t) > \frac{1}{K}.
\end{align}
Moreover, $\mathrm{T}_0$ satisfies  
\begin{align}\label{formule T0} 
    \mathrm{T}_0=\inf\left\{ \mathrm{T}> 0 : m(\mathrm{T})<K \right\} = \sup\left\{ \mathrm{T}> 0 : m(\mathrm{T})=K \right\}
\end{align}
and $\mathcal{M}(\mathrm{T})> \mathcal{M}(\infty)$ for all $\mathrm{T}> \mathrm{T}_0$.
\end{proposition}

\begin{proof}
Let us verify separately the two possible cases established in Lemma \ref{Lemma Olcay BFR}.
		
\begin{enumerate} 
\item[$a)$] Suppose $r(0)\leq \frac{1}{\fa}$. The mean residual life function $m$ is continuous and satisfies $m(0) = \fa > K$. According to Lemma \ref{Lemma Olcay BFR}, $m$ is decreasing on $[0,\infty)$. Therefore, for the existence of some $\mathrm{T} > 0$ such that $m(\mathrm{T})<K$, it is necessary and sufficient that 
\[
\displaystyle \lim\limits_{t\rightarrow\infty} m(t) < K.
\]
This condition is equivalent to $\lim\limits_{t\rightarrow \infty}	r(t) > \frac{1}{K}$ since applying L'H\^opital's rule, it can be verified that
\[
\displaystyle \lim\limits_{t\rightarrow\infty} m(t) = \frac{1}{\lim\limits_{t\rightarrow \infty}	r(t)}.
\]

Finally, by continuity and monotonicity of $m$, we deduce that
\begin{align*} 
	\mathrm{T}_0 = \sup\{\mathrm{T}> 0 : m(\mathrm{T})=K\},
\end{align*}
and that $m(\mathrm{T}) \geq K$ for $\mathrm{T}\leq \mathrm{T}_0$, while $m(\mathrm{T})< K$ for $\mathrm{T}>\mathrm{T}_0$. Consequently, $\mathcal{M}(\mathrm{T})> \mathcal{M}(\infty)$ for all $\mathrm{T} > \mathrm{T}_0$.

\item[$b)$] On the other hand, if $r(0) > \frac{1}{\fa}$, then there exists a value $c \in (0,\tmin)$ such that $m$ is increasing on $[0,c]$ and decreasing on $[c,\infty)$.

Let us examine the behavior of $m$ on each interval.

    \begin{itemize}
        \item On $[0,c]$, the function $m$ is increasing and satisfies
        \[
        m(\mathrm{T}) \ge m(0) = \fa > K, \quad \text{for } 0 \le \mathrm{T} \le c ,
        \]
    and, by Proposition \ref{prop: beneficial maintenance}, $\mathcal{M}(\mathrm{T}) < \mathcal{M}(\infty)$ for all $\mathrm{T} \in [0,c]$.

    \item On $[c,\infty)$, the function $m$ is decreasing and $m(c)\geq K$. The reasoning is then analogous to Case $a)$, restricting the analysis to the interval $[c,\infty)$.
\end{itemize}	

\end{enumerate}
\end{proof}

Part $a)$ of the previous proof corresponds to the particular case $X_1 \in \mathrm{DMRL}$. This immediately leads to the following corollary.

\begin{corollary} \label{corol: dmrl} 
Let $\Delta \mu>0$ and $X_1\in\mathrm{DMRL}$. Then $\mathrm{T}_0>0$ exists if and only if \eqref{condition T0} holds. Moreover, $\mathrm{T}_0$ satisfies \eqref{formule T0}, and $\mathcal{M}(\mathrm{T})> \mathcal{M}(\infty)$ for all $\mathrm{T}> \mathrm{T}_0$.
\end{corollary}

Proposition \ref{prop: imrl} and Corollary \ref{corol: dmrl} together show that preventive maintenance is advantageous when the lifetime of the main unit exhibits aging behavior in terms of the mean residual life function, whereas it yields no benefit when the lifetime distribution exhibits anti-aging behavior in the same sense.

We now examine the existence of the optimal preventive maintenance time $\mathrm{T}^*$, as defined in \eqref{eq: formule T*}. According to Proposition \ref{prop : T_0}, maintenance is beneficial if and only if $\lim\limits_{t\rightarrow \infty} r(t) > \frac{1}{K}$, which thus establishes a necessary condition for the existence of an optimal maintenance time.

The following proposition ensures the existence of $\mathrm{T}^*$ and derives an explicit expression for it in terms of the function
\begin{align}\label{eq : varphi} 
\varphi(t) = r(t)\int_{0}^{t}\overline{F}(x)\,\mathrm{d}x + \overline{F}(t) = \fa\, r(t) - m'(t) \overline{F}(t), \quad \text{for all }t\geq 0.
\end{align}

\begin{proposition} \label{prop : T opt}
Let $\Delta\mu>0$, $X_1\in\mathrm{BFR}$, and $\lim\limits_{t\rightarrow \infty}	r(t) > \frac{1}{K}$. Then, there exists an optimal maintenance time $\displaystyle \mathrm{T}^*$, which satisfies the equality 
\begin{align*}
\mathrm{T}^* = \inf\left\{\mathrm{T} \geq \max\{\mathrm{T}_0,\tmin \} : \varphi(\mathrm{T}) = \frac{\mu_1}{\Delta\mu} \right\}.
\end{align*}
If $r(t)$ increases strictly on $[\tmin,\infty)$, then $\displaystyle \mathrm{T}^*$ is the unique solution to $\varphi(\mathrm{T}) = \frac{\mu_1}{\Delta\mu}$.
\end{proposition}

\begin{proof} 		
To determine the optimal maintenance time, we consider the maximization of the function $\displaystyle \mathcal{M}(\mathrm{T})$ for $\mathrm{T}> 0$. Suppose that $r$ is differentiable. Then 
\begin{equation}\label{eq : sgn of h'}
	\mathcal{M}'(\mathrm{T}) =_{\mathrm{sg}}{\frac{ \mu_1  }{\Delta \mu}-\varphi(\mathrm{T})} \quad \text{and} \quad \varphi'(\mathrm{T}) =  r'(\mathrm{T})\int_{0}^{\mathrm{T}}\overline{F}(x)\,\mathrm{d}x, 
\end{equation} 
where $\varphi$ is defined in \eqref{eq : varphi}. 

From the monotonicity intervals of $r$ and the expression for $\varphi'$, it follows that $\varphi$ decreases on $[0,\tmin]$ and increases on $[\tmin,\infty)$. If $r$ is not differentiable, the monotonicity of $\varphi$ can be deduced from the identity 
\[
\varphi(t) - \varphi(s) = (r(t)-r(s)) \int_0^s \overline{F}(x)\,\mathrm{d}x + \int_s^t (r(t)-r(x)) \overline{F} (x) \,\mathrm{d}x, \quad 0\leq s\leq t. 
\]
It holds that
\[
\varphi(\tmin) \leq \varphi(0)=1 \quad \text{and} \quad \lim_{t\to\infty} \varphi(t) > \frac{ \mu_1  }{\Delta \mu} > 1.
\]
By continuity of $\varphi$, there exists $\mathrm{T}>\tmin$ such that $\varphi(\mathrm{T}) = \frac{ \mu_1 }{\Delta \mu}$. If $r$ increases strictly on $[\tmin,\infty)$, then $\displaystyle \mathrm{T}^*$ is the unique solution to $\varphi(\mathrm{T}) = \frac{ \mu_1 }{ \Delta\mu}$; otherwise,
\[
\mathrm{T}^* = \inf\left\{\mathrm{T} > \tmin: \varphi(\mathrm{T}) = \frac{ \mu_1  }{ \Delta\mu} \right\}.
\]
Since $\varphi$ increases on $[\tmin,\infty)$, equation \eqref{eq : sgn of h'} implies that $\mathcal{M}'(\mathrm{T}) \geq 0$ if $\mathrm{T}\leq \mathrm{T}^*$, whereas $\mathcal{M}'(\mathrm{T}) \leq 0$ if $\mathrm{T}\geq \mathrm{T}^*$. Thus, $\mathcal{M}$ increases on $[0,\mathrm{T}^*]$ and decreases on $[\mathrm{T}^*,\infty)$, proving that $\mathrm{T}^*$ maximizes $\mathcal{M}$.

Recall that $\mathrm{T}_0 = \inf{\mathrm{T}>0 : \mathcal{M}(\mathrm{T}) > \mathcal{M}(\infty)}$. To verify that the optimal time $\mathrm{T}^*$ satisfies $\mathrm{T}^* \ge \mathrm{T}_0$, note that $\mathcal{M}(\mathrm{T}^*)$ is the global maximum of $\mathcal{M}$, so that $\mathcal{M}(\mathrm{T}^*) \ge \mathcal{M}(\mathrm{T})$ for all $\mathrm{T}>0$. In particular, $\mathcal{M}(\mathrm{T}^*) \ge \mathcal{M}(\infty)$, which implies $\mathrm{T}^* \ge \mathrm{T}_0$.
\end{proof}

The case $\tmin=0$ in Proposition \ref{prop : T opt} corresponds to the IFR setting, for which analogous results have been established in other preventive maintenance models; see, for instance, \cite{ZHONG2014405}.

\subsection{Upside-down bathtub-shaped hazard rate}\label{Sec 3.2}

Motivated by the results obtained for the case $X_1\in\mathrm{BFR}$, it is natural to investigate whether analogous conclusions can be obtained in the dual setting. Suppose that the hazard rate $r$ exhibits an upside-down bathtub shape.
Once again, we rely on a result from \cite{Olcay1995}. In particular, Lemma \ref{Lemma Olcay UBFR} addresses the monotonicity of the mean residual life in the case where $X_1\in\mathrm{UBFR}$.

\begin{lemma}[Theorem 2 in \cite{Olcay1995}]\label{Lemma Olcay UBFR}
Let $X_1 \in \mathrm{UBFR}$.
\begin{enumerate}
	\item[$a)$] If $r(0)\geq \frac{1}{\fa}$, then $X_1\in\mathrm{IMRL}$.
	\vspace{0em}
	\item[$b)$] If $r(0)< \frac{1}{\fa}$, then there exists $c \in (0,\tmax)$ such that $m$ is decreasing on $ [0,c]$ and increasing on $[c,\infty)$.
\end{enumerate}
\end{lemma}

An immediate consequence of Proposition \ref{prop: imrl} together with Lemma \ref{Lemma Olcay UBFR} is the following result.

\begin{corollary}
Let $X_1\in\mathrm{UBFR}$. If $r(0)\geq \frac{1}{\fa}$, then $\mathcal{M}(\mathrm{T}) < \mathcal{M}(\infty)$, for any $\mathrm{T}>0$.
\end{corollary}

The proposition below establishes conditions ensuring the existence of a threshold $\mathrm{T}_0$, such that the benefit of preventive maintenance, depending on \eqref{condition T0}, may arise either for all $\mathrm{T}>\mathrm{T}_0$ or only over a finite interval.

\begin{proposition}\label{prop: T0 - UBFR}
Consider $\displaystyle \Delta \mu > 0$, $X_1\in\mathrm{UBFR}$, and assume $r(0)< \frac{1}{\fa}$. Let $c$ denote the minimizer of $m$ given in Lemma \ref{Lemma Olcay UBFR}. Then a time $\mathrm{T}_0 \in(0,c)$ exists if and only if 
\[
m(c) < K.
\]
Moreover:
\begin{enumerate}
\item[$a)$] If \eqref{condition T0} holds, then $\mathcal{M}(\mathrm{T})>\mathcal{M}(\infty)$ for all $\mathrm{T} > \mathrm{T}_0$.
\item[$b)$] If \eqref{condition T0} does not hold, then there exists $\mathrm{T}_1>c$ such that $\mathcal{M}(\mathrm{T})>\mathcal{M}(\infty)$ for $\mathrm{T} \in (\mathrm{T}_0, \mathrm{T}_1)$ and $\mathcal{M}(\mathrm{T}) \leq \mathcal{M}(\infty)$ for $\mathrm{T} \geq \mathrm{T}_1$.
\end{enumerate}
\end{proposition}

\begin{proof}
By Proposition \ref{prop : T_0 general}, $\mathrm{T}_0>0$ exists if and only if $m(c)<K$. 
Moreover, since 
\[ 
\mathrm{T}_0=\inf\{\mathrm{T}>0 : m(\mathrm{T})<K\}
\] 
and $m(c)<K$, it follows that $\mathrm{T}_0<c$.

As $m$ is decreasing on $[\mathrm{T}_0,c]$, we obtain
\[
m(\mathrm{T}) < K \quad \text{for all } \mathrm{T} \in (\mathrm{T}_0,c].
\]

\begin{enumerate}
\item[$a)$] Because $m$ is increasing on $[c,\infty)$,
\[
m(\mathrm{T}) \le \lim_{t\to\infty} m(t) < K 
\quad \text{for all } \mathrm{T} \in [c,\infty),
\]
and therefore $\mathcal{M}(\mathrm{T})>\mathcal{M}(\infty)$ for all $\mathrm{T} > \mathrm{T}_0$.

\item[$b)$] Since $m(c)<K$ and $\displaystyle \lim_{t\to\infty} m(t) \ge K$, there exists $\mathrm{T}>c$ such that $m(\mathrm{T})=K$. Define
\[
\mathrm{T}_1=\inf\{\mathrm{T}>c : m(\mathrm{T})=K\}.
\]
Given that $m$ is increasing on $[c,\infty)$,
\[
m(\mathrm{T})<K \quad \text{for } \mathrm{T}\in[c,\mathrm{T}_1), 
\qquad
m(\mathrm{T})\ge K \quad \text{for } \mathrm{T}\ge \mathrm{T}_1 .
\]
Thus,  $\mathcal{M}(\mathrm{T})>\mathcal{M}(\infty)$ for $\mathrm{T}\in(\mathrm{T}_0,\mathrm{T}_1)$ and 
$\mathcal{M}(\mathrm{T})\le \mathcal{M}(\infty)$ for $\mathrm{T}\ge \mathrm{T}_1$.
\end{enumerate}
\end{proof}

Note that $\mathrm{T}_0<c$, so it lies in the region where the mean residual life function decreases while the hazard rate increases. Nevertheless, in case $a)$, for all $\mathrm{T}>\mathrm{T}_0$, preventive maintenance enhances the system lifetime, even within the interval $[\tmax,\infty)$, where the mean residual life increases and the hazard rate decreases.

The necessary and sufficient condition for maintenance to be beneficial highlights a contrast between the two settings: for $X_1\in\mathrm{BFR}$, it is determined by the limit of the mean residual life at infinity, whereas for $X_1\in\mathrm{UBFR}$, it is governed by its minimum value.

The condition $m(c)<K$ is equivalent to $r(c)>\frac{1}{K}$. Since 
\[
\displaystyle c = \arg\min_{\mathrm{T}>0} m(\mathrm{T}),
\] 
we have $m'(c)=0$. Using the identity $m'(t) = m(t)r(t)-1$, it follows that $m(c)r(c)=1$. Therefore,
\[
m(c)<K \Longleftrightarrow r(c)>\frac{1}{K}.
\]
This equivalence plays the same role in Proposition \ref{prop: T opt - UBFR} as condition \eqref{condition T0} does in Proposition \ref{prop : T_0}. Hence, $r(c)>\frac{1}{K}$ is a necessary condition for the existence of $\mathrm{T}^*$.

\begin{proposition} \label{prop: T opt - UBFR}
Let $\Delta\mu>0$, $X_1\in\mathrm{UBFR}$, $r(0)< \frac{1}{\fa}$, and $r(c) > \frac{1}{K}$. Then $\displaystyle \mathrm{T}^*$ exists and satisfies 
\begin{align*}
\mathrm{T}^* = \inf\left\{\mathrm{T} \in (\mathrm{T}_0,c) : \varphi(\mathrm{T}) = \frac{ \mu_1}{ \Delta\mu} \right\}.
\end{align*}
If $r(t)$ increases strictly on $(0,c]$, then $\displaystyle \mathrm{T}^*$ is the unique solution to $\varphi(\mathrm{T}) = \frac{\mu_1}{ \Delta\mu}$.
\end{proposition}

\begin{proof} 		
Recall that for $\mathrm{T}> 0$,
\begin{equation}\label{eq : signe of h'}
	\mathcal{M}'(\mathrm{T}) =_{\mathrm{sg}}{\frac{ \mu_1 }{\Delta \mu}-\varphi(\mathrm{T})}; \quad \quad \varphi'(\mathrm{T}) =  r'(\mathrm{T})\int_{0}^{\mathrm{T}}\overline{F}(x)\,\mathrm{d}x. 
\end{equation} 
We restrict attention to the case where $r$ is differentiable, since the general case follows from the same argument as in Proposition \ref{prop : T opt}. From the monotonicity intervals of $r$ and the expression for $\varphi'$, it follows that $\varphi$ increases on $[0,\tmax]$ and decreases on $[\tmax,\infty)$. 

Moreover,
\[
\varphi(0)=1 \quad \text{and} \quad \varphi(c) = \fa \, r(c) - 1 > \frac{ \mu_1 }{\Delta \mu}>1.
\]
By continuity of $\varphi$, there exists $\mathrm{T}\in(0,c)$ such that $\varphi(\mathrm{T}) = \frac{ \mu_1 }{\Delta \mu}$. If $r$ is strictly increasing on $(0,c]$, then $\displaystyle \mathrm{T}^*$ is the unique solution to $\varphi(\mathrm{T}) = \frac{ \mu_1 }{ \Delta\mu}$; otherwise, we define 
\[
\mathrm{T}^* = \inf\left\{\mathrm{T} \in (0,c): \varphi(\mathrm{T}) = \frac{ \mu_1 }{ \Delta\mu} \right\}.
\]
Since $\varphi$ increases on $[0,\tmax]$, equation \eqref{eq : signe of h'} implies that $\mathcal{M}'(\mathrm{T}) \geq 0$ for $\mathrm{T}\leq \mathrm{T}^*$, while $\mathcal{M}'(\mathrm{T}) \leq 0$ for $\mathrm{T}^*\leq \mathrm{T}\leq \tmax$. Hence $\mathrm{T}^*$ is a local maximizer of $\mathcal{M}$.

To prove that $\mathrm{T}^*$ is the global maximizer, we distinguish the two cases of Proposition \ref{prop: T0 - UBFR}.

\begin{itemize}
\item[$a)$] Suppose that $\displaystyle \lim_{t\to\infty} r(t)> \frac{1}{K}$. Since $\varphi$ decreases on $[\tmax,\infty)$, we obtain 
\[
\displaystyle \varphi(\mathrm{T}) \geq \lim_{t\to\infty} \varphi(t) >  \frac{ \mu_1 }{ \Delta\mu}, \quad \mathrm{T}\geq \tmax.
\]
Thus, $\mathcal{M}'(\mathrm{T}) \leq 0$ for all $\mathrm{T}\geq \tmax$, and therefore $\mathrm{T}^*$ is the optimal maintenance time.

\item[$b)$] Suppose now that $\displaystyle \lim_{t\to\infty} r(t)\leq \frac{1}{K}$. From Proposition \ref{prop: T0 - UBFR}, there exists 
\[
\mathrm{T}_1 = \inf \{ \mathrm{T} > c : m(\mathrm{T}) = K \},
\]
and maintenance is beneficial only on $(\mathrm{T}_0,\mathrm{T}_1)$. If $\mathrm{T}_1 \leq \tmax$, then the previous argument already implies that $\mathrm{T}^*$ is optimal.

Assume instead that $\mathrm{T}_1> \tmax$. Since $m'(\mathrm{T}_1)\geq 0$ and $m(\mathrm{T}_1)=K$, we obtain
\[
\varphi(\mathrm{T}_1) = \frac{\fa(m'(\mathrm{T}_1)+1)}{m(\mathrm{T}_1)} - m'(\mathrm{T}_1) \overline{F}(\mathrm{T}_1) = \frac{\mu_1}{\Delta \mu} + m'(\mathrm{T}_1) \left( \frac{\mu_1}{\Delta \mu} - \overline{F}(\mathrm{T}_1) \right).
\]
Given that $\overline{F}(\mathrm{T}_1)\leq 1 < \frac{\mu_1}{\Delta \mu}$ it follows that $\varphi(\mathrm{T}_1) \geq \frac{\mu_1}{\Delta \mu}$. Since $\varphi$ decreases on $[\tmax,\infty)$, we get
\[
\varphi(\mathrm{T}) \geq \varphi(\mathrm{T}_1)\geq \frac{\mu_1}{\Delta \mu}, \quad \tmax \leq \mathrm{T}\leq \mathrm{T}_1.
\]
Therefore $\mathcal{M}'(\mathrm{T})\leq 0$ on $[\tmax,\mathrm{T}_1)$ and we conclude that $\mathrm{T}^*$ is the optimal maintenance time.
\end{itemize}
\end{proof}

\section{Numerical examples}\label{Sec 4}
This section presents numerical examples that illustrate the theoretical results of Sections \ref{Sec 3.1} and \ref{Sec 3.2}. Assuming that repair and maintenance times are exponentially distributed, we analyze both the maintenance threshold time and the optimal maintenance time for two lifetime distributions of the main unit, one exhibiting a bathtub-shaped hazard rate and the other an upside-down bathtub-shaped hazard rate. In each case, the hazard rate and the difference $\mathcal{M}(\mathrm{T})-\mathcal{M}(\infty)$ are plotted, providing numerical confirmation of the theoretical findings.

Throughout this section, we assume that $Y_j$ follows an exponential distribution with rate $\gamma_j$, for $j=1,2$.

\subsection{Bathtub-shaped hazard rate}
Let $Z_1$ and $Z_2$ be independent Weibull random variables with shape parameters $k_1<1$ and $k_2>1$, and scale parameters $\beta_1>0$ and $\beta_2>0$, respectively. That is,
\[
Z_i \sim \mathrm{Weibull}(\beta_i, k_i), \quad i=1,2. 
\]
Define $X_1 = Z_1 \wedge Z_2$. This construction admits two interpretations. First, $X_1$ can be viewed as the lifetime of a system consisting of two Weibull units arranged in series, so that failure occurs when the first unit fails. Second, it can be interpreted as a reduced-form model in which the main unit fails either due to early failures or due to aging and usage.
 
Since the hazard rate of the minimum of independent random variables equals the sum of their hazard rates, the hazard rate of $X_1$ is given by
\begin{align}\label{hazard rate bfr example}
r(t) = \frac{k_1}{\beta_1} \left(\frac{t}{\beta_1}\right)^{k_1 - 1} + \frac{k_2}{\beta_2} \left(\frac{t}{\beta_2 } \right)^{k_2 - 1}, \qquad t>0. 
\end{align}
It is well known (see, for example, \cite{Marshall2007}, Chapter 4.D) that the hazard rate defined in \eqref{hazard rate bfr example} is bathtub-shaped, as illustrated in Figure \ref{fig:hazard BFR}. Moreover, since 
\[
\displaystyle \lim_{t\to\infty} r(t) = \infty,
\] 
the existence of both the threshold time $\mathrm{T}_0$ and the optimal maintenance time $\mathrm{T}^*$ follows from Propositions \ref{prop : T_0} and \ref{prop : T opt}.

A straightforward calculation yields
\[
\mathcal{M}(\mathrm{T}) = \frac{1}{\lambda} \left(1 + \frac{\displaystyle (\lambda + \gamma_1)(\lambda + \gamma_2) \int_{0}^{\mathrm{T}} \exp\left\{-\left( \frac{t}{\beta_1}\right)^{k_1} - \left( \frac{t}{\beta_2}\right)^{k_2} \right\} \,\mathrm{d}t}{\displaystyle \lambda + \gamma_2 - \left(\gamma_2- \gamma_1\right) \,\exp\left\{-\left(\frac{\mathrm{T}}{\beta_1}\right)^{k_1} - \left( \frac{\mathrm{T}}{\beta_2}\right)^{k_2} \right\} }\right), \quad \mathrm{T}>0,
\]
and
\[
\mathcal{M}(\infty) = \frac{1}{\lambda}\left(1+(\lambda + \gamma_1 )\displaystyle \int_{0}^{\infty} \exp\left\{-\left( \frac{t}{\beta_1}\right)^{k_1} - \left( \frac{t}{\beta_2}\right)^{k_2} \right\} \,\mathrm{d}t\right).
\]
Using Proposition \ref{prop : T_0}, the threshold time $\mathrm{T}_0$ is computed numerically as the unique solution of \[m(t) = \frac{\fa(\gamma_2 - \gamma_1)}{\lambda+\gamma_2}, \quad \text{where} \,\,\,
m(t) = \frac{ \displaystyle\int_t^\infty \exp\left\{-\left( \frac{x}{\beta_1}\right)^{k_1} - \left( \frac{x}{\beta_2}\right)^{k_2} \right\} \,\mathrm{d} x}{\exp\left\{-\left( \frac{t}{\beta_1}\right)^{k_1} - \left( \frac{t}{\beta_2}\right)^{k_2} \right\}}. 
\]

From Proposition \ref{prop : T opt}, the optimal maintenance time $\mathrm{T}^*$ is the unique solution to
\[
\left( \frac{k_1}{\beta_1} \left(\frac{t}{\beta_1}\right)^{k_1 - 1} + \frac{k_2}{\beta_2} \left(\frac{t}{\beta_2 } \right)^{k_2 - 1} \right)\int_{0}^{t} \mathrm{e}^{-\left( \frac{x}{\beta_1}\right)^{k_1} - \left( \frac{x}{\beta_2}\right)^{k_2}} \mathrm{d}x + \mathrm{e}^{-\left( \frac{t}{\beta_1}\right)^{k_1} - \left( \frac{t}{\beta_2}\right)^{k_2}} = \frac{\lambda+\gamma_2}{\gamma_2 - \gamma_1}.
\]
Figure \ref{fig:diff BFR} plots the difference $\mathcal{M}(\mathrm{T}) - \mathcal{M}(\infty)$. Across all numerical experiments, $\mathrm{T}_0$ lies in the interval where the hazard rate is increasing. Graphically, it can be observed that the sharper the curvature of the hazard rate, the larger the values of both $\mathrm{T}_0$ and $\mathrm{T}^*$.

\captionsetup[subfigure]{labelformat=empty}
\begin{figure}[htbp]
\centering
\begin{subfigure}{0.48\textwidth}
    \centering
    \includegraphics[width=\linewidth]{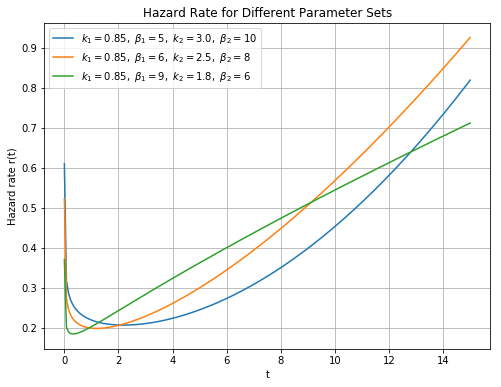}
    \caption{$a)$ Hazard rate function defined in \eqref{hazard rate bfr example}. }
    \label{fig:hazard BFR}
\end{subfigure}
\hfill
\begin{subfigure}{0.48\textwidth}
    \centering
    \includegraphics[width=\linewidth]{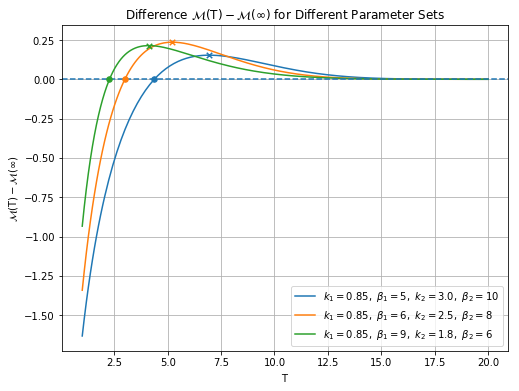}
    \caption{$b)$ Difference $\mathcal{M}(\mathrm{T})- \mathcal{M}(\infty)$.}
    \label{fig:diff BFR}
\end{subfigure}
\caption{$\mathrm{BFR}$ example with $\gamma_1 = 0.001$, $\gamma_2 = 4$ and $\lambda =1.0$.}
\end{figure}

\subsection{Upside-down bathtub-shaped hazard rate}

Let $Z_1$ and $Z_2$ be independent exponential random variables with different rates $\beta_1>0$ and $\beta_2>0$, respectively. That is,
\[
Z_i \sim \mathrm{Exponential}(\beta_i), \quad i=1,2. 
\]
Define $X_1 = Z_1 \vee Z_2$. The main unit can be  interpreted as a system composed of two subunits arranged in parallel, where failure occurs only when both subunits have failed. The hazard rate of $X_1$ is
\begin{align}\label{hazard rate ubfr example}
r(t) = \frac{ \beta_1 \,\mathrm{e}^{\beta_2 t} + \beta_2\, \mathrm{e}^{\beta_1 t} - \beta_1-\beta_2}{ \mathrm{e}^{\beta_1 t} + \mathrm{e}^{\beta_2 t} - 1 }, \qquad t>0. 
\end{align}
As illustrated in Figure \ref{fig:hazard ubfr}, the hazard rate has an upside-down bathtub shape. Moreover $r(0)=0$ and \[\displaystyle \lim_{t\to\infty} r(t) = \min\{\beta_1, \beta_2\}.\] Setting $\gamma_1 = 0.01$, $\gamma_2 = 6$, and $\lambda =0.1$, we verify graphically that the condition $m(c)<K$ holds, where $c = \arg\min m$. This condition ensures the existence of both $\mathrm{T}_0$ and $\mathrm{T}^*$, as established in Propositions \ref{prop: T0 - UBFR} and \ref{prop: T opt - UBFR}. 

We obtain
\[
\mathcal{M}(\mathrm{T}) = \frac{1}{\lambda} \left(1 + \frac{\displaystyle (\lambda + \gamma_1) (\lambda + \gamma_2) \left( \frac{(1-\mathrm{e}^{-\beta_1 \mathrm{T} })}{\beta_1} + \frac{(1 - \mathrm{e}^{-\beta_2 \mathrm{T} })}{\beta_2} - \frac{(1-\mathrm{e}^{-(\beta_1 + \beta_2)\mathrm{T} })}{\beta_1+\beta_2} \right)}{ \lambda + \gamma_2 - \left( \gamma_2 - \gamma_1 \right) \,\left(\mathrm{e}^{-\beta_1 \mathrm{T} }+ \mathrm{e}^{-\beta_2 \mathrm{T} } - \mathrm{e}^{-(\beta_1 + \beta_2)\mathrm{T} }\right) }\right)\text{ ,}
\]
and
\[
\mathcal{M}(\infty) = \frac{1}{\lambda}\left(1+(\lambda + \gamma_1 ) \left( \frac{1}{\beta_1} + \frac{1}{\beta_2}  - \frac{1}{\beta_1+\beta_2} \right) \right).
\]
From Proposition \ref{prop: T0 - UBFR}, the threshold time $\mathrm{T}_0$ is the solution to
\[
\frac{ \frac{\mathrm{e}^{\beta_2 t}}{\beta_1} + \frac{\mathrm{e}^{\beta_1 t}}{\beta_2} - \frac{1}{\beta_1 + \beta_2} }{\mathrm{e}^{\beta_1 t}+ \mathrm{e}^{\beta_2 t} - 1} = \frac{(\gamma_2 - \gamma_1)}{\lambda+\gamma_2} \left( \frac{1}{\beta_1} + \frac{1}{\beta_2} - \frac{1}{\beta_1 + \beta_2} \right)
\]
Moreover, using Proposition \ref{prop: T opt - UBFR}, the optimal maintenance time $\mathrm{T}^*$ is the unique solution to
\[
r(t) \left( \frac{(1-\mathrm{e}^{-\beta_1 t})}{\beta_1} + \frac{(1 - \mathrm{e}^{-\beta_2 t})}{\beta_2} - \frac{(1-\mathrm{e}^{-(\beta_1 + \beta_2)t})}{\beta_1+\beta_2} \right) + \mathrm{e}^{-\beta_1 t} + \mathrm{e}^{-\beta_2 t} - \mathrm{e}^{-(\beta_1 + \beta_2)t } = \frac{\lambda+\gamma_2}{(\gamma_2 - \gamma_1)}.
\]
Numerical values of $\mathrm{T}_0$ and $\mathrm{T}^*$ for different parameter sets are shown in Figure \ref{fig:diff ubfr}. In particular, for the green curve, preventive maintenance is beneficial only over a finite time interval. This behavior is consistent with the fact that 
\[
\displaystyle \lim_{t\to\infty} r(t) = \min\{\beta_1, \beta_2\}< \frac{1}{K}=\frac{(\gamma_2 - \gamma_1)}{\lambda+\gamma_2} \left( \frac{1}{\beta_1} + \frac{1}{\beta_2} - \frac{1}{\beta_1 + \beta_2} \right).
\]
\begin{figure}[htbp]
\centering
\begin{subfigure}{0.48\textwidth}
    \centering
    \includegraphics[width=\linewidth]{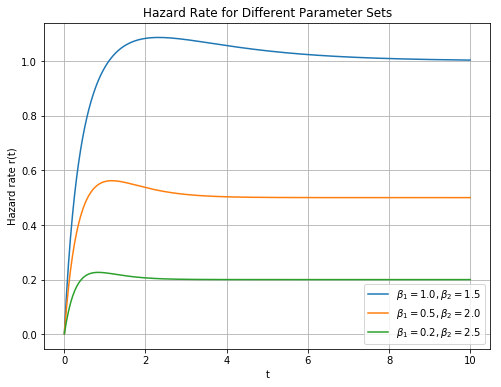}
    \caption{$a)$ Hazard rate function defined in \eqref{hazard rate ubfr example}.}
    \label{fig:hazard ubfr}
\end{subfigure}
\hfill
\begin{subfigure}{0.48\textwidth}
    \centering
    \includegraphics[width=\linewidth]{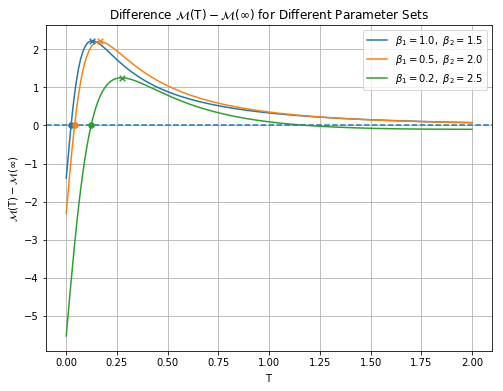}
    \caption{$b)$ Difference $\mathcal{M}(\mathrm{T})- \mathcal{M}(\infty)$.}
    \label{fig:diff ubfr}
\end{subfigure}
\caption{$\mathrm{UBFR}$ example with $\gamma_1 = 0.01$, $\gamma_2 = 6$ and $\lambda =0.1$.}

\end{figure}

\section{Stochastic orderings}  \label{Sec 5}
In this section, we establish necessary and sufficient conditions for stochastic comparisons between two independent priority standby redundant systems with preventive maintenance. From this point onward, let $X_{ik}$ denote the lifetime of unit $k$ in system $i$, where $i,k=1,2$. For maintenance and repair times, we use the notation $Y_{ij}$, where $j=1$ corresponds to the repair time and $j=2$ to the preventive maintenance time in system $i=1,2$. We denote by $\lambda_i$ the rate of the exponential distribution of $X_{i2}$ for $i=1,2$. In each case, the corresponding survival functions, expected values, hazard rates, and related quantities are denoted analogously and in accordance with the notations established in previous sections.

To ensure that maintenance is beneficial in both systems, we assume that
\[
\Delta_{i} \mu = \mu_{i1} - \mu_{i2} = \mathbb{E}[X_{i2} \wedge Y_{i1}] - \mathbb{E}[X_{i2} \wedge Y_{i2}]>0, \quad i=1,2.
\]

\subsection{Stochastic ordering for the comparison of mean lifetimes}

We assume that maintenance in both systems is scheduled at time $\mathrm{T}>0$, and we denote by $\mathcal{M}_i(\mathrm{T})$ the mean lifetime of system $i$, $i=1,2$. The next proposition summarizes the main comparison result. 

\begin{proposition}\label{prop MTTF comp} 
Let $\mathrm{T}>0$. Assume that $X_{11}\geq_{\mathrm{st}} X_{21}$, $\lambda_1\leq\lambda_{2}$, and $Y_{1j}\leq_{\mathrm{Lt}} Y_{2j}$ for $j=1,2$. Then $\displaystyle \mathcal{M}_{1}(\mathrm{T}) \geq \mathcal{M}_{2}(\mathrm{T})$.
\end{proposition} 

The proof of Proposition \ref{prop MTTF comp} is a direct consequence of the three results established below.

\begin{proposition}\label{prop: comp M simplest case}
Let $\mathrm{T}>0$, $\lambda_1=\lambda_2$, and $Y_{1j} =_{\mathrm{st}} Y_{2j}$ for $j=1,2$. If $X_{11} \geq_{\mathrm{st}} X_{21}$, then $\mathcal{M}_1(\mathrm{T}) \geq \mathcal{M}_2(\mathrm{T})$.
\end{proposition}

\begin{proof}
Since $\lambda_1=\lambda_2$ and $Y_{1j} =_{\mathrm{st}} Y_{2j}$, we have $\mu_{1j} = \mu_{2j}$, and we denote this common value by $\mu_j$, for $j=1,2$. From \eqref{mean time - priority and PM system}, it follows that $\mathcal{M}_1(\mathrm{T})\geq \mathcal{M}_2(\mathrm{T})$ if and only if 
\[
\mu_1 \left( \int_0^\mathrm{T} \overline{F}_1 (t) \mathrm{d}t - \int_0^\mathrm{T} \overline{F}_2 (t) \mathrm{d}t \right) -\Delta\mu \left( \overline{F}_2 (\mathrm{T}) \int_0^\mathrm{T} \overline{F}_1 (t) \mathrm{d}t - \overline{F}_1 (\mathrm{T}) \int_0^\mathrm{T} \overline{F}_2 (t) \mathrm{d}t \right) \geq 0,
\]
where $\Delta \mu = \mu_1 - \mu_2$. Since $\mu_1 > \Delta \mu \geq 0$, it suffices to show that
\[
\int_0^\mathrm{T} \overline{F}_1 (t) \mathrm{d}t - \int_0^\mathrm{T} \overline{F}_2 (t) \mathrm{d}t \geq \overline{F}_2 (\mathrm{T}) \int_0^\mathrm{T} \overline{F}_1 (t) \mathrm{d}t - \overline{F}_1 (\mathrm{T}) \int_0^\mathrm{T} \overline{F}_2 (t) \mathrm{d}t.
\]
The last inequality is equivalent to 
\begin{equation}\label{eq: DMTFR order}
\frac{\int_0^\mathrm{T} \overline{F}_1 (t)\, \mathrm{d}t}{ F_1(\mathrm{T})} \geq \frac{\int_0^\mathrm{T} \overline{F}_2 (t) \, \mathrm{d}t}{ F_2(\mathrm{T})}.
\end{equation}
This condition can be interpreted in terms of the \textit{mean time to failure until replacement function} (see, e.g., \cite{kayid2016laplace, Corujo_Valdes_2022}), defined by 
\[
g(t) = \frac{\int_0^t \overline{F} (x)\, \mathrm{d}x}{ F(t)}, \quad t>0.
\]
Thus, the result follows if the mean time to failure until replacement function of $X_{11}$ is greater than or equal to that of $X_{21}$ at time $\mathrm{T}$. This condition follows directly  from $X_{11} \geq_{\mathrm{st}} X_{21}$.
\end{proof}

The following proposition establishes the comparison when the rates of the standby unit lifetimes are ordered, while the remaining variables are identically distributed. 

\begin{proposition}\label{prop: comp M lambda}
Let $\mathrm{T}>0$. Assume that $X_{11}=_{\mathrm{st}} X_{21}$ and $Y_{1j}=_{\mathrm{st}} Y_{2j}$ for $j=1,2$. If $\lambda_1\leq \lambda_2$, then $\displaystyle \mathcal{M}_1(\mathrm{T}) \geq \mathcal{M}_2(\mathrm{T})$.
\end{proposition}

\begin{proof}
Let $F$ denote the common distribution function of $X_{11}$ and $X_{21}$. Also, let $Y_{1j}=_{\mathrm{st}} Y_{2j} =_{\mathrm{st}} Y_j$, and let $\overline{G}_j$ denote the survival function of $Y_j$ for $j=1,2$. Define $\mu_{ij}= \mathbb{E}[X_{i2} \wedge Y_{j}]$ for $i,j=1,2$. From \eqref{mean time - priority and PM system}, it suffices to show that
\[
\lambda_1\left(\mu_{11}\,F(\mathrm{T})  + \mu_{12}\,\overline{F}(\mathrm{T}) \right) \leq \lambda_2 \left(\mu_{21}\,F(\mathrm{T})  + \mu_{22}\,\overline{F}(\mathrm{T})  \right).
\]
For $i,j=1,2$, we have
\[
\lambda_i \,\mu_{ij} = \lambda_i \int_{0}^{\infty} \mathrm{e}^{-\lambda_i t}\, \overline{G}_j(t)\,\mathrm{d}t
= \mathbb{E}\!\left[\overline{G}_j\!\left(X_{i2}\right)\right].
\]
Since $\lambda_1\leq \lambda_2$, we have $X_{12} \geq_{\mathrm{st}} X_{22}$. By Characterization \eqref{st} of the usual stochastic order, and since $\overline{G}_j$ is decreasing, it follows that
\begin{align*}
\lambda_1 \,\mu_{1j} = \mathbb{E}\!\left[\overline{G}_j\!\left(X_{12}\right)\right] \leq \mathbb{E}\!\left[\overline{G}_j\!\left(X_{22}\right)\right] = \lambda_2 \,\mu_{2j}.
\end{align*}
Then the result holds.
\end{proof}

If the unit lifetimes are identically distributed, reducing repair and maintenance times increases the mean time to system failure.

\begin{proposition}\label{mttf comparison 3}
Let $\mathrm{T}>0$. Assume that $X_{11}=_{\mathrm{st}} X_{21}$ and $\lambda_1=\lambda_{2}$. If $Y_{1j}\leq_{\mathrm{Lt}} Y_{2j}$ for $j=1,2$, then $\displaystyle \mathcal{M}_{1}(\mathrm{T}) \geq \mathcal{M}_{2}(\mathrm{T})$.
\end{proposition}
\begin{proof}
Let $F$ denote the common distribution function of $X_{11}$ and $X_{21}$. It suffices to show that 
\[
\mu_{11}\,F(\mathrm{T})  + \mu_{12}\,\overline{F}(\mathrm{T}) \leq \mu_{21}\,F(\mathrm{T}) + \mu_{22}\,\overline{F}(\mathrm{T}).
\]
A sufficient condition for this inequality is that $\mu_{1j} \leq \mu_{2j}$ for $j=1,2$, which can be rewritten as
\[
\widehat{\overline{G}}_{1j}(\lambda)\leq \widehat{\overline{G}}_{2j}(\lambda), \quad j=1,2.
\]
This condition holds whenever $Y_{1j}\leq_{\mathrm{Lt}} Y_{2j}$ for $j=1,2$.
\end{proof}

Although the assumption of Laplace transform ordering guarantees the result in the previous proposition, it is worth noting that the weaker condition $\mu_{1j} \leq \mu_{2j}$, $j=1,2$ is already sufficient. This observation will be useful in subsequent results where comparisons are established via the Laplace transform order.

\subsection{Stochastic ordering for the comparison of maintenance threshold times}

We assume that the lifetimes of both standby units follow exponential distributions with the same parameter, and that the repair and maintenance times are identically distributed. For each system $i$, let $\fa_i = \mathbb{E}\left[X_{i1}\right]$ and $\mathrm{T}_{i0}$ denote the corresponding threshold time, $i=1,2$. 

\begin{proposition}\label{prop: T0 order}
Assume that $\lambda_1 = \lambda_2$ and $Y_{1j} =_{\mathrm{st}} Y_{2j}$ for $j=1,2$. If $X_{11} \geq_{\mathrm{mrl}} X_{21}$ and $\fa_1 = \fa_2$, then $\mathrm{T}_{10} \geq \mathrm{T}_{20}$.
\end{proposition}

\begin{proof}
Under the assumptions, we have $K_1 = K_2 = K$, where $K_i$ is the constant defined in \eqref{formule K} for system $i$, $i=1,2$. Moreover, the relation $X_{11} \geq_{\mathrm{mrl}} X_{21}$ implies that, for any $\mathrm{T}>0$, 
\[
m_1(\mathrm{T}) < K \quad \Rightarrow \quad m_2(\mathrm{T}) < K.
\]
Consequently,
\[
\left\{ \mathrm{T}>0 : m_1(\mathrm{T}) < K \right\} \subseteq \left\{ \mathrm{T}>0 : m_2(\mathrm{T}) < K \right\}.
\]
Taking infima on both sets yields
\[
\inf \left\{ \mathrm{T}>0 : m_1(\mathrm{T}) < K \right\} \geq \inf \left\{ \mathrm{T}>0 : m_2(\mathrm{T}) < K \right\},
\]
which proves the result.
\end{proof}

Proposition \ref{prop: T0 order} shows that even when the principal units have the same mean lifetime, a larger mean residual life leads to a larger maintenance threshold time. Hence, the system whose main unit ages more favorably allows preventive maintenance to be postponed.

\subsection{Stochastic ordering for the comparison of optimal maintenance times}

Let $\mathrm{T}_i^{*}$ denote the optimal maintenance time of system $i$, and let $K_i$ denote the constant associated with system $i$ and defined in \eqref{formule K}, for $i=1,2$. The following results establish conditions under which
\[
\mathrm{T}_1^{*} \geq \mathrm{T}_2^{*}.
\]

\begin{proposition}\label{prop: T* comp BFR}
Assume that $X_{11}=_{\mathrm{st}} X_{21}$, that their common distribution belongs to the $\mathrm{BFR}$ class, and that their hazard rate satisfies
\begin{align}\label{condition for T* comparison}
\lim_{t\rightarrow \infty} r(t) >  \frac{1}{K_1 \wedge K_2}.
\end{align}
Under these conditions, the optimal preventive maintenance times $\mathrm{T}_1^*$ and $\mathrm{T}_2^*$ exist. Moreover, $\mathrm{T}_1^{*} \geq \mathrm{T}_2^{*}$ if and only if $\mu_{21}\mu_{12} \geq \mu_{11} \mu_{22}$.
\end{proposition}

\begin{proof}
Let $\overline{F}$, $r$, $m$, and $\fa$ denote the survival function, the hazard rate, the mean residual life function, and the expectation of $X_{11}$ (and thus also of $X_{21}$). By \eqref{condition for T* comparison} and Proposition \ref{prop : T opt}, the optimal preventive maintenance times $\mathrm{T}_1^*$ and $\mathrm{T}_2^*$ exist. Furthermore,
\[
\mathrm{T}_i^* = \inf\left\{\mathrm{T} \ge \tmin : \varphi(\mathrm{T}) = \frac{\mu_{i1}}{\Delta_i \mu}\right\},
\]
where $\Delta_i \mu = \mu_{i1}-\mu_{i2}$. Using the identity
\[
\varphi(t)=\fa\, r(t)-m'(t)\overline F(t), \quad t\geq 0,
\]
the condition $\varphi(\mathrm{T}) = \frac{\mu_{i1}}{\Delta_i \mu}$ can be rewritten as
\[
r(t)-\frac{m'(t)\overline F(t)}{\fa}=\frac{1}{K_i},
\]
where $K_i = \fa \left( 1 - \frac{\mu_{i2}}{\mu_{i1}}\right)$.

A direct calculation shows that $K_1 \le K_2$ if and only if $\mu_{21}\mu_{12} \ge \mu_{11}\mu_{22}$. Since $\varphi$ is increasing on $[\tmin,\infty)$, the function $t \mapsto \frac{\varphi(t)}{\fa} = r(t)-\frac{m'(t)\overline F(t)}{\fa}$ is also increasing on $[\tmin,\infty)$. Therefore, $\mathrm{T}_i^*$ is decreasing with $K_i$. Consequently,
\[
\mathrm{T}_1^* \ge \mathrm{T}_2^* \quad \Leftrightarrow \quad K_1 \le K_2 \quad \Leftrightarrow  \quad \mu_{21}\mu_{12} \ge \mu_{11}\mu_{22}. 
\]
\end{proof}

If $\mu_{11}\leq \mu_{21}$ and $\mu_{12} \geq \mu_{22}$, then condition $\mu_{21}\mu_{12} \ge \mu_{11}\mu_{22}$ is satisfied. If, in addition, $X_{12} =_{\mathrm{st}} X_{22}$ and their common distribution is exponential with rate $\lambda$, then the inequalities
\[
\mu_{11}\leq \mu_{21}, \quad \mu_{21} \geq \mu_{22} \quad \text{and} \quad \Delta_i \mu \geq 0, \, i=1,2,
\]
can be rewritten as
\[
\widehat{\overline{G}}_{21}(\lambda)\geq \widehat{\overline{G}}_{11}(\lambda) \geq \widehat{\overline{G}}_{12}(\lambda)\geq \widehat{\overline{G}}_{22}(\lambda). 
\]
This condition holds whenever $Y_{22} \leq_{\mathrm{Lt}} Y_{12} \leq_{\mathrm{Lt}}  Y_{11}\leq_{\mathrm{Lt}} Y_{21}$.

\begin{corollary}\label{corol: T* comp LT BFR}
Under the conditions of Proposition \ref{prop: T* comp BFR}, if $X_{12} =_{\mathrm{st}} X_{22}$ and 
\[
Y_{22} \leq_{\mathrm{Lt}} Y_{12} \leq_{\mathrm{Lt}} Y_{11}\leq_{\mathrm{Lt}} Y_{21},
\] 
then $\mathrm{T}_1^{*} \geq \mathrm{T}_2^{*}$.
\end{corollary}

When corresponding units in both systems have identical lifetime distributions and the hazard rate of the main units has a bathtub shape, the following interpretation holds. If repairs in system 1 are faster than in system 2 in the Laplace transform order, then the optimal maintenance time can be larger for system 1, even if its maintenance times are longer than those of system 2. 

The proof of the following proposition is analogous to that of Proposition \ref{prop: T* comp BFR} and is therefore omitted.

\begin{proposition} \label{prop: T* comp UBFR}
Assume that $X_{11}=_{\mathrm{st}} X_{21}$, that their common distribution belongs to the $\mathrm{UBFR}$ class, and that their common hazard rate satisfies
\[
r(0)\leq \frac{1}{\fa} \quad \text{and} \quad r(c) > \frac{1}{K_1 \wedge K_2},
\]
where $c$ denotes the minimizer of their mean residual life function. Under these conditions, the optimal preventive maintenance times $\mathrm{T}_1^*$ and $\mathrm{T}_2^*$ exist. Moreover, $\mathrm{T}_1^{*} \geq \mathrm{T}_2^{*}$ if and only if $\mu_{21}\mu_{12} \geq \mu_{11} \mu_{22}$.
\end{proposition}

The following result can be obtained using the same arguments as in Corollary \ref{corol: T* comp LT BFR}, and its interpretation is analogous.

\begin{corollary}
Under the conditions of Proposition \ref{prop: T* comp UBFR}, if $X_{12} =_{\mathrm{st}} X_{22}$ and 
\[
Y_{22} \leq_{\mathrm{Lt}} Y_{12} \leq_{\mathrm{Lt}}  Y_{11}\leq_{\mathrm{Lt}} Y_{21},
\]
then $\mathrm{T}_1^{*} \geq \mathrm{T}_2^{*}$.
\end{corollary}
\section{Conclusions}\label{coclusiones}
The main contribution of this study is the derivation of necessary and sufficient conditions for the existence of an optimal preventive maintenance time in a two-unit priority standby system with repair, when the main unit follows a bathtub-shaped or an upside-down bathtub-shaped hazard rate. Expressions to determine a threshold time beyond which maintenance improves the mean time to system failure and to determine the optimal maintenance time are provided  in terms of the hazard rate and the mean residual life function.  Moreover, stochastic ordering techniques allowed comparisons between two independent systems under this framework, including their mean lifetimes, threshold times and optimal maintenance times. Numerical examples illustrate the impact of the hazard rate shape on preventive maintenance decisions.

A possible direction for future research is to extend the present results to a more general framework in which the standby unit follows a general lifetime distribution.

{\small
\bibliographystyle{apalike}
\bibliography{references}}

\end{document}